\documentclass[final,12pt]{colt2026} 

\usepackage{algorithm}
\usepackage{algorithmic}
\usepackage{enumerate}
\usepackage{indentfirst}
\usepackage{xcolor}
\usepackage{makecell}

\DeclareMathOperator{\E}{\mathbb{E}}
\DeclareMathOperator{\Prob}{\mathbb{P}}
\DeclareMathOperator{\argmin}{argmin}
\def\beh#1{{\color{black}#1}}

\newcommand{\Nc}{\mathcal{N}}
\newcommand{\Oc}{\mathcal{O}}
\newcommand{\R}{\mathbb{R}}
\newcommand{\Z}{\mathbb{Z}^+}
\newcommand{\eps}{\varepsilon}

\newcommand{\Sp}{\mathbb{S}}
\newcommand{\Ss}{\mathcal{S}}
\newcommand{\proj}{\Pi_\mathcal{S}}

\newcommand{\x}{\boldsymbol{x}}
\newcommand{\y}{\boldsymbol{y}}
\newcommand{\z}{\boldsymbol{z}}
\newcommand{\bxi}{\boldsymbol{\xi}}
\newcommand{\bu}{\boldsymbol{u}}
\newcommand{\bv}{\boldsymbol{v}}
\newcommand{\step}{\eta}
\newcommand{\tauit}{\tau_i(t)}
\newcommand{\grad}{\boldsymbol{g}}
\newcommand{\pfi}{\Tilde{\grad}_i}

\newcommand{\pg}{\boldsymbol{h}}
\newcommand{\Ps}{P_\mathcal{S}}

\newtheorem{assumption}{Assumption}

\title[Convex and Non-convex Federated Learning with Stale Stochastic Gradients]{Convex and Non-convex Federated Learning with Stale Stochastic Gradients: Diminishing Step Size is All You Need}
\usepackage{times}

\coltauthor{%
 \Name{Xinran Zheng} \Email{xinran22@illinois.edu}\\
 \addr University of Illinois Urbana-Champaign
 \AND
 \Name{Tara Javidi} \Email{tjavidi@ucsd.edu}\\
 \addr University of California San Diego
 \AND
 \Name{Behrouz Touri} \Email{touri1@illinois.edu}\\
 \addr University of Illinois Urbana-Champaign%
}

\begin{document}

\maketitle

\begin{abstract}%
    We propose a general framework for distributed stochastic optimization under delayed gradient models. In this setting, $n$ local agents leverage their own data and computation to assist a central server in minimizing a global objective composed of agents' local cost functions. Each agent is allowed to transmit stochastic—potentially biased and delayed—estimates of its local gradient. While a prior work has advocated delay-adaptive step sizes for stochastic gradient descent (SGD) in the presence of delays, we demonstrate that a pre-chosen diminishing step size is sufficient and matches the performance of the adaptive scheme. Moreover, our analysis establishes that diminishing step sizes recover the optimal SGD rates for nonconvex and strongly convex objectives.
\end{abstract}


\section{Introduction}

Distributed optimization and federated learning have become a core primitive in modern learning and decision systems, where data and computation are inherently decentralized across many devices, organizations, or geographic locations. 
In a typical server–client architecture, a central coordinator seeks to minimize a global objective formed by aggregating local cost functions, while $n$ agents contribute updates computed from their own data. This paradigm enables scalability and privacy-preserving training, but also introduces new algorithmic and theoretical challenges because communication~\cite{alistarh2017qsgd}, computation~\cite{recht2011hogwild}, and data access~\cite{mcmahan2017communication,fang2022communication} are all imperfect.

This paper focuses on gradient-descent–based methods due to their simplicity and broad applicability, which has been extensively studied in \cite{nedic2009distributed,tsianos2012consensus,nedic2014distributed,tatarenko2017non}.
In practical distributed deployments, however, two phenomena are hard to ignore.

First, the gradient information transmitted by agents is typically stochastic—arising from data sampling, communication compression/quantization, and local inexact computations—and most existing analyses adopt the simplifying assumption that these stochastic gradients are unbiased \cite{alistarh2017qsgd,sra2016adadelay,ghadimi2016mini}. Nevertheless, biased gradient surrogates naturally appear in practice; a prominent example is zeroth-order optimization, where agents can only query function values and must construct gradient estimates via random perturbations, which are generally biased \cite{tang2023zeroth,zheng2023zeroth,fang2022communication,malladi2023fine}.

Second, agent updates are frequently delayed because of stragglers, intermittent connectivity, or asynchronous communication~\cite{reisizadeh2022straggler,yang2022federated}; as a result, the server may apply stale gradient information computed at earlier iterates. In prior work, such delays are typically assumed to be a random variable that is uniformly bounded by a constant~\cite{tang2023zeroth,zheng2023zeroth,zheng2024general}, or has i.i.d.\ exponential distribution across time~\cite{reisizadeh2022straggler}. A weaker assumption of scaled delay is addressed by an algorithm employing delay-adaptive step-sizes in~\cite{sra2016adadelay}.
When constraints are present, projected variants of SGD further couple these effects with feasibility considerations, making it essential to understand how bias and delay jointly impact convergence guarantees.

In this work we study constrained stochastic gradient descent under a delayed, approximate gradient model with a mild delay condition. Specifically, agents may transmit stochastic—possibly biased—estimates of their local gradients, and the server updates using gradient information that is scaled-delayed, meaning that the gradient used at time $t$ is computed no earlier than time $\kappa t$ for some $\kappa\in(0,1)$, with the delay also having bounded second moment. To the best of our knowledge, this is the first work to analyze constrained SGD with biased stochastic gradient estimators under this scaled delay model. A key message of our results is that one does not need delay-adaptive schemes: a standard pre-chosen diminishing step size already suffices to achieve optimal performance (up to logarithmic factors) under the same delay assumptions considered in previous delay-adaptive analyses~\cite{sra2016adadelay}.

\begin{table}[t]
\centering
\caption{Assumption comparison with related works.}
\label{tab:compare_assumption}
\renewcommand{\arraystretch}{1.2}
\setlength{\tabcolsep}{6pt}
\small

\begin{tabular}{|c|c|c|c|}
\hline
Work & \textbf{Domain of $\x$} & \textbf{Stochastic gradient} & \textbf{Delay} \\
\hline
\textbf{Ours} & constrained & biased & scaled  \\
\hline
\scriptsize{\cite{ghadimi2016mini}} & constrained & unbiased & 0  \\
\hline
\scriptsize{\cite{sra2016adadelay}} & constrained & unbiased & scaled \\
\hline
\scriptsize{\cite{zheng2023zeroth}} & unconstrained & biased & bounded \\
\hline
\scriptsize{\cite{zheng2024general}} & unconstrained & biased & bounded \\
\hline
\end{tabular}
\end{table}

\begin{table}[t]
\centering
\caption{Performance comparison with related works.}
\label{tab:compare_performance}
\renewcommand{\arraystretch}{1.2}
\setlength{\tabcolsep}{6pt}
\small

\begin{tabular}{|c|c|c|c|}
\hline
Work & \makecell{\textbf{Non-convex}\\ $\frac{1}{T+1}\!\sum_{t=0}^{T}{\E[\|\nabla f(\x(t))\|^2]}$} & \makecell{\textbf{Strongly convex}\\ $\E[\|\x(T)-\x^*\|^2]$} & \makecell{\textbf{Convex}\\ $\E[f(\Tilde{\x}(T))]-f^*$} \\
\hline
\textbf{Ours} &  $\Oc(1)$ for projected gradient & $\Oc(\frac{1}{T})$ & $\Oc(\frac{\log T}{\sqrt{T}})$ \\
\hline
\scriptsize{\cite{ghadimi2016mini}} & $\Oc(1)$ for projected gradient &  & $\Oc(\frac{\log T}{\sqrt{T}})$ \\
\hline
\scriptsize{\cite{sra2016adadelay}} &   &  & $\Oc(\frac{1}{\sqrt{T}})$ \\
\hline
\scriptsize{\cite{zheng2023zeroth}} &  $\Oc(\frac{\log T}{\sqrt{T}})$ &  &  \\
\hline
\scriptsize{\cite{zheng2024general}} &   & $\Oc(\frac{1}{T})$ &  \\
\hline
\end{tabular}
\end{table}

\paragraph{Contributions}

We study a general framework for distributed optimization with delayed, approximate stochastic gradients under a mild delay condition. Specifically, we assume that the delay is scaled (i.e., the gradient information used at time $t$ is computed not earlier than time $\kappa t$) and has bounded second moment. We show that for \textbf{non-convex} functions, the convergence performance for $\frac{1}{T+1}\!\sum_{t=0}^{T}{\E[\|\pg(t)\|^2]}$, where $\pg(t)$ denotes the projected gradient, matches that of classical unbiased SGD without delay. For \textbf{strongly convex} functions, the mean squared error achieves an $\Oc(\frac{1}{T})$ rate, again matching the best-known rates for classical SGD without delay. Furthermore, for general \textbf{convex} functions, we obtain an error bound of $\Oc\left(\frac{\log T}{\sqrt{T}}\right)=\Oc(\frac{1}{T^{\frac{1}{2}-\eps}})$ for arbitrarily small $\eps>0$, which matches the error bound $\Oc\left(\frac{1}{\sqrt{T}}\right)$ of using delay-adaptive step sizes up to a logarithmic factor.

\paragraph{Notations}
We denote the set of non-negative integers $\{0,1,\ldots\}$ by $\Z$, the real numbers by $\R$, and the vector space of $d$-dimensional real-valued vectors by $\R^d$. We use bold lower-case letters to denote vectors. We use $\boldsymbol{I}_d$ to denote the identity matrix of dimension $d$. We use $\Nc(0, \boldsymbol{\Sigma})$ to denote the multi-dimensional Gaussian distribution with zero mean and covariance $\boldsymbol{\Sigma}$. We use $[n]$ to denote $\{1, \ldots, n\}$. We use $\| \cdot \|$ to denote the standard Euclidean norm, and $\langle\cdot,\cdot\rangle$ to denote the standard  inner product given by $\langle \x,\y\rangle=\x^T\y$. We let $\Sp_{d-1}:=\{ \x\in\R^d:\|\x\|=1 \}$ denote the unit sphere, and $\text{Unif}(\Sp_{d-1})$ denote the uniform distribution over $\Sp_{d-1}$. For any random variable $X$, we use $\sigma(X)$ to denote the $\sigma$-algebra generated by $X$.

\section{Problem Formulation}

We consider optimizing a decomposable cost function conducted by a central server and $n$ agents or workers, where they work jointly to optimize a global function, which is the sum of $n$ local functions as
\begin{align}
\label{eqn:distributed_opt}
    \min_{\x \in \Ss} f(\x):=\sum_{i=1}^n f_i(\x),
\end{align}
where $\Ss\subset\R^d$ is a closed convex set. We assume that each local function $f_i$ (and any gradient information thereof) is accessible only to agent $i$. The central server has no direct access to the $f_i$’s; instead, it maintains the global decision variable $\x$ and, at each time $t$, broadcasts $\x(t)$ to all agents and requests gradient information evaluated at $\x(t)$. The gradient returned by agent~$i$, denoted $\grad_i(\x,\bxi)$, depends on $\x$ and a random variable $\bxi$ that is independent of $\x$. In addition, due to computation and communication latency, the server may receive and apply a stale estimate computed at an earlier iterate.

\section{The Main Results}

In this section, we introduce a general framework of a class of algorithms to solve problem~(\ref{eqn:distributed_opt}). These algorithms operate on the discrete-time instants $t\in\Z$.

At time $t=0$, the central server generates an initial guess at an arbitrary point $\x(0)=\x_0\in \Ss$. Then, for $t\in\Z$, the central server sends $\x(t)$ and the time stamp $t$ to each agent. Meanwhile, each agent computes or estimates the gradient of $f_i$ at some of the $\x$'s that it has received, and sends the gradient information together with the corresponding original time stamps back to the central server. At the same time, the central server receives gradient information from each agent, and keeps track of the latest gradient information $\grad_i(\x(\tauit), \bxi(\tauit))$ from agent $i$. Here, $\x(\tauit)$, which was sent out by the central server at time $\tauit$, denotes the latest received gradient information from agent $i$ by time $t$. If the central server has never received any gradient information from agent $i$, we let $\grad_i(\x(\tauit), \bxi(\tauit))=0$ and $\tauit=-1$. Then, the central server computes an approximate gradient for the global cost function as
\begin{align}\label{eqn:grad_t}
    \grad(t)=\sum_{i=1}^n\grad_i(\x(\tauit), \bxi(\tauit)),
\end{align}
and updates the optimization variable as
\begin{align}
\label{eqn:maindyn}
    \x(t+1)=\proj[\x(t)-\step(t)\grad(t)],
\end{align}
where $\proj:\R^d\rightarrow\Ss$ is the projection operator on the set $\Ss$, and $\step(t)$ is an appropriate step size.

\subsection{Analysis}
Here, we present our main results regarding the convergence guarantee for the general framework. In order to do so, we make the following assumptions on objective functions, gradient estimators and communication delay.

The first assumption is on objective functions.

\begin{assumption}[Assumption on Objective Functions]
\label{asm:func}
    We assume ${f^*:=\inf_{\x \in \Ss} f(\x)>-\infty}$, and $f$ is continuously differentiable. 
    Further, for $i\in [n]$, $f_i: \Ss \rightarrow \R$ is  $L$-smooth, i.e., for all $\x,\y \in \Ss$, $${\|\nabla f_i(\x)-\nabla f_i(\y)\| \leq L\| \x-\y \|}.$$ 
\end{assumption}

The second assumption is on the gradient estimators. \beh{To discuss this, for the estimates (stochastic process) $\{\x(t)\}_{t\in\Z}$ generated by \eqref{eqn:maindyn}, let us define the random process $\{\pfi(t)\}_{t\in\Z}$ to be $\pfi(t):=\E[\grad_i(\x(t), \bxi(t))|\sigma(\x(t))]$ for each $i\in[n]$.} The following assumption establishes the relationship between $\grad_i(\cdot)$ and $\nabla f_i(\cdot)$.

\begin{assumption}[Assumption on Gradient Estimators]
\label{asm:grad_est}
    We assume that for the random process $\{\x(t)\}_{t\in\Z}$ generated by \eqref{eqn:maindyn},  there exists $G\in(0,\infty)$ such that
 ${\E[\|\grad_i(\x(t), \bxi(t))\|^2|\sigma(\x(t))]\leq G}$ for all $i \in [n]$.  Further, for the expected gradient estimates $\{\pfi(t)\}_{t\in\Z}$, we assume that for any $t_1,t_2\in\Z$, we have
    \begin{align*}
        \|\pfi(t_1)-\pfi(t_2)\| \leq L\| \x(t_1)-\x(t_2) \|.
    \end{align*}
    Finally, we assume that $\|\pfi(t)-\nabla f_i(\x(t))\| \leq q(t)$
    for some non-negative sequence $\{q(t)\}_{t\in \Z}$. We refer to $q(t)$ as the gradient estimator bias.
\end{assumption}


\begin{remark}
    As discussed in a closely related work \cite{zheng2024general}, Assumption \ref{asm:grad_est} is satisfied by a broad class of gradient estimators. It includes  unbiased stochastic gradient (with ${q(t)=0}$), as well as many zeroth-order gradient estimators using random perturbation. For instance, in \cite{tang2023zeroth,zheng2023zeroth}, two-point gradient estimator with Gaussian perturbation is used, i.e., the gradient estimate of $f_i$ at $\x(t)$ is
    \begin{align}
    \label{eqn:2point_grad_est}
        &\grad_i(\x(t),\z(t))=\frac{f_i(\x(t)+u(t)\z(t))-f_i(\x(t)-u(t)\z(t))}{2u(t)}\z(t),
    \end{align}
    where $u(t)$ is the smoothing radius, and ${\z(t)\sim\Nc(0,\boldsymbol{I}_d)}$. It has been shown that the bias of this gradient estimator is bounded by $L\sqrt{d}u(t)$, enabling control of the convergence rate via the choice of $u(t)$. Another example is the two-point gradient estimator based on smoothing over a unit ball~\cite{shamir2017optimal}
    \begin{align*}
        &\grad_i(\x(t),\z(t))=d\frac{f_i(\x(t)+u(t)\z(t))-f_i(\x(t)-u(t)\z(t))}{2u(t)}\z(t),
    \end{align*} 
    where $\z(t)\sim \text{Unif}(\Sp_{d-1})$, whose bias is bounded by $Lu(t)$.

\end{remark}

In the following analysis, for brevity, we write $\grad_i(\x, \bxi)$ as $\grad_i(\x)$ whenever the dependence on the random variable $\bxi$ is not needed.
Under Assumption~\ref{asm:grad_est}, we can promptly derive that the second moment of the global gradient estimator is bounded. 

\begin{lemma}
\label{lemma:bound_of_g}
    For $\grad(t)$ given in \eqref{eqn:grad_t}, we have $\E[\|\grad(t)\|^2]\leq n^2G$ for all $t\in \Z$.
\end{lemma}

Our last assumption is on communication delay.

\begin{assumption}[Assumption on Communication Delay]
\label{asm:delay}
    We assume that the delay $t-\tauit$ is independent across time and agents. We also assume that the second moment of delay is bounded, i.e., there exists a $0<C<\infty$ s.t. $\E[(t-\tauit)^2]\leq C$ for all $t\in\Z$ and all $i\in [n]$. Further, we assume that for all $t\in \Z$, there is a $\kappa\in(0,1)$ s.t. $t-\tauit\leq(1-\kappa)t$, i.e., $\tauit\geq\kappa t$.
\end{assumption}

\begin{remark}
    Note that Assumption \ref{asm:delay} is much weaker than the assumption that all delays are upper bounded by a constant~\cite{tang2023zeroth,zheng2023zeroth,zheng2024general}. Assumption \ref{asm:delay} is inspired by the delay model described in Assumption 2.5(B) of \cite{sra2016adadelay}. As we will show later, with this weaker assumption, we can achieve the same convergence rate with respect to time $t$ as under the bounded delay assumption.
\end{remark}

With Assumptions \ref{asm:func}-\ref{asm:delay}, we obtain the key lemma underpinning our analysis, which quantifies the impact of delay on the expected distance between $\x(t)$ and $\x(\tauit)$.

\begin{lemma}
\label{lemma:effect_of_delay}
    Suppose that Assumption \ref{asm:func}-\ref{asm:delay} hold. Let $p(t)=\step(\lceil\kappa t\rceil)$. Then,  for any ${i\in [n]}$ and any $t\in\Z$, we have 
    \begin{align*}
        \E[\|\x(t)-\x(\tauit)\|^2]&\leq n^2GCp^2(t),
        \end{align*}
        and
     \begin{align*}
        \E[\|\pfi(t)-\pfi(\tauit)\|^2]&\leq n^2GCL^2p^2(t).
    \end{align*}
    Similarly, $\E[\|\nabla f(\x(t))-\nabla f(\x(\tauit))\|^2]\leq n^2GCL^2p^2(t).$
\end{lemma}

\begin{proof} 
We have
\begin{align*}
    &\E[\|\x(t)-\x(\tauit)\|^2]
    =\E[\|\sum_{\tau=1}^{t-\tauit}(\x(t-\tau+1)-\x(t-\tau))\|^2]\\
    =&\E[\|\sum_{\tau=1}^{t-\tauit}\big(\proj[\x(t-\tau)-\step(t-\tau)\grad(t-\tau)]-\proj[\x(t-\tau)]\big)\|^2]\\
    \leq&\E[(\sum_{\tau=1}^{t-\tauit}\big\|\proj[\x(t-\tau)-\step(t-\tau)\grad(t-\tau)]-\proj[\x(t-\tau)]\big\|)^2]\\
    \stackrel{\rm{(a)}}{\leq}&\E[(\sum_{\tau=1}^{t-\tauit}\|\step(t-\tau)\grad(t-\tau)\|)^2]
    \stackrel{\rm{(b)}}{\leq}\E[\sum_{\tau=1}^{t-\tauit}\step^2(t-\tau)\sum_{\tau=1}^{t-\tauit}\|\grad(t-\tau)\|^2]\cr 
    \stackrel{\rm{(c)}}{\leq}&\E\left[\E[\sum_{\tau=1}^{t-\tauit}\step^2(t-\tau)\sum_{\tau=1}^{t-\tauit}\|\grad(t-\tau)\|^2|\tauit]\right]
    \end{align*}
where (a) follows from the nonexpansive property of projection (see Corollary 2.2.3 in~\cite{nesterov2018lectures}), (b) follows from the Cauchy-Schwartz Inequality, and (c) follows from the law of total expectation. But, 
\begin{align*}
    &\E\left[\E[\sum_{\tau=1}^{t-\tauit}\step^2(t-\tau)\sum_{\tau=1}^{t-\tauit}\|\grad(t-\tau)\|^2|\tauit]\right]
    \stackrel{\rm{(a)}}{\leq}n^2G\E[(t-\tauit)\sum_{\tau=1}^{t-\tauit}\step^2(t-\tau)]
    \\
    \stackrel{\rm{(b)}}{\leq}&n^2G\sqrt{\E[(t-\tauit)^2]\E[(\sum_{\tau=1}^{t-\tauit}\step^2(t-\tau))^2]}
    \stackrel{\rm{(c)}}{\leq}n^2G\sqrt{C\E[\big((t-\tauit)\step^2(\tauit)\big)^2]}\\
    \stackrel{\rm{(d)}}{\leq}&n^2G\sqrt{C\E[(t-\tauit)^2]\step^4(\lceil\kappa t\rceil)}
    \stackrel{\rm{(e)}}{\leq}n^2GC\step^2(\lceil\kappa t\rceil)
    =n^2GCp^2(t),
\end{align*}
where (a) follows from Lemma~\ref{lemma:bound_of_g},  (b) follows from the Cauchy-Schwartz Inequality, (c), (d), and (e) use Assumption~\ref{asm:delay}, and (c) and (d) use the diminishing property of the step size sequence.

For the second inequality, by Assumption \ref{asm:grad_est}, we have
\begin{align*}
    \E[\|\pfi(t)-\pfi(\tauit)\|^2]
    \leq L^2\E[\|\x(t)-\x(\tauit)\|^2]
    \leq n^2GCL^2p^2(t). 
\end{align*}

The last inequality follows similarly as the above inequality.
\end{proof}
\vspace{-0.5cm}
\subsubsection{Main Results}

Now we are ready to present our main results that characterize the convergence rate of \eqref{eqn:maindyn} under our assumptions. The first result establishes the convergence properties of~\eqref{eqn:maindyn} for general (non-convex) function $f$.

\begin{theorem}\label{thm:nonconvex}
    Let $\{\pg(t)\}_{t\in\Z}$ be defined by $$\pg(t):=\frac{1}{\step(t)}\left(\x(t)-\proj[\x(t)-\step(t)\nabla f(\x(t))]\right).$$
    Let the step size sequence $\{\step(t)\}$ be a decreasing sequence s.t. $\lim_{t\rightarrow\infty}\frac{\step(\lceil\kappa t\rceil)}{\step(t)}=\bar{\eta}>0$ where $\bar{\eta}>0$ is a constant and ${\sum_{t=0}^{\infty}\step(t)=\infty}$. Under Assumptions \ref{asm:func}-\ref{asm:delay}, we have
    \begin{align*}
        \frac{1}{T+1}\!\sum_{t=0}^{T}{\E[\|\pg(t)\|^2]}=\Oc\left(\frac{\sum_{t=0}^T\left(\step(t)+\step^2(t)+\step(t)q^2(t)\right)}{T\step(T)}\right).
    \end{align*}
    In particular, for $\step(t)=\frac{\step_0}{(t+1)^\alpha}$ for some constant $\step_0>0$ and $\alpha\in(0,1)$, and any non-increasing $\{q(t)$\}, we have $\frac{1}{T+1}\sum_{t=0}^{T}{\E[\|\pg(t)\|^2]}=\Oc(1)$.
\end{theorem}

\begin{remark}
    Similar to \cite{ghadimi2016mini}, here we address constrained optimization; accordingly, in the nonconvex setting we measure stationarity via the projected gradient mapping $\pg(t)$ rather than $\nabla f(\x(t))$. 
    In addition, we can rewrite our result in the same way as Corollary 2 in~\cite{ghadimi2016mini} as follows. If $s$ is a nonnegative integer-valued random variable with the probability mass function
   $\Prob(s=t)=\frac{\step(t)}{\sum_{t=0}^T\step(t)}$
    for any $t=0,\ldots,T$, we have
    \begin{align*}
        \E_s[\|\pg(s)\|^2]=\Oc\left(\frac{\sum_{t=0}^T\left(\step(t)+\step^2(t)+\step(t)q^2(t)\right)}{\sum_{t=0}^T\step(t)}\right).
    \end{align*}
    This result shows that $\E[\|\pg(s)\|^2]$ converges to a neighborhood around $0$, whose radius is governed by the coefficient multiplying $\sum_{t=0}^T \step(t)$ in the numerator. As shown in the proof below, this coefficient equals $\frac{5}{2}n^2G$. This is consistent with Corollary 2 in~\cite{ghadimi2016mini}, which states that the neighborhood size scales proportionally with the second moment of the stochastic gradient and decreases inversely with the batch size.
\end{remark}

\begin{proof}
    (\textit{Proof of Theorem \ref{thm:nonconvex}})
    

    For $\x\in\Ss$, $\bv\in\R^d$, and $t\in \Z$, we denote \beh{$\Ps(\x,\bv,t)=\frac{1}{\step(t)}\left(\x-\proj[\x-\step(t)\bv]\right)$}.
    Thus, our goal is to establish a bound for ${\pg(t)=\Ps(\x(t),\nabla f(\x(t)),t)}$.

    As proven in  \cite{ghadimi2016mini}, $\Ps(\x,\bv,t)$ has the following properties: for any ${\bv,\bv_1,\bv_2\in\R^d}$ and $t\in \Z$,
    \begin{align}
        \label{eqn:Ps_prop1}
        &\|\Ps(\x,\bv,t)\|^2 \leq \langle\bv, \Ps(\x,\bv,t)\rangle \leq \|\bv\|^2,\\
        \label{eqn:Ps_prop2}
        &\|\Ps(\x,\bv_1,t)-\Ps(\x,\bv_2,t)\| \leq \|\bv_1-\bv_2\|.
    \end{align}

     With an abuse of notation, for brevity, let us denote $\Ps(\x(t),\grad(t),t)$ by $\Ps(\x(t),\grad(t))$. Using $L$-smooth condition of $f$, \eqref{eqn:maindyn} implies
    \begin{align*}
        f(\x(t+1)) \leq f(\x(t))-\step(t)\langle\nabla f(\x(t)),\Ps(\x(t),\grad(t))\rangle+\frac{L}{2}\step^2(t)\|\Ps(\x(t),\grad(t))\|^2.
    \end{align*}
    Taking the expected value of both sides, we have
    \begin{align}
    \label{eqn:expected_iter}
        \E[f(\x(t+1))] 
        \leq\E[f(\x(t))]+\step(t)\E[-\langle\nabla f(\x(t)),\Ps(\x(t),\grad(t))\rangle]+\frac{L}{2}\step^2(t)\E[\|\Ps(\x(t),\grad(t))\|^2].
    \end{align}
    For the second term in the right hand side of \eqref{eqn:expected_iter}, by adding and subtracting intermediate terms, we have 
    \begin{align}
        \label{eqn:4inner_products}
        &\E[-\langle\nabla f(\x(t)),\Ps(\x(t),\grad(t))\rangle]\cr
        =&\E[-\langle\nabla f(\x(t)),\Ps(\x(t),\grad(t))-\Ps\left(\x(t),\sum_{i=1}^n\pfi(\tauit)\right)\rangle]\cr
        &+\E[-\langle\nabla f(\x(t)),\Ps\left(\x(t),\sum_{i=1}^n\pfi(\tauit)\right)-\Ps\left(\x(t),\sum_{i=1}^n\pfi(t)\right)\rangle]\cr
        &+\E[-\langle\nabla f(\x(t)),\Ps\left(\x(t),\sum_{i=1}^n\pfi(t)\right)-\Ps\left(\x(t),\sum_{i=1}^n\nabla f_i(\x(t))\right)\rangle]\cr
        &+\E[-\langle\nabla f(\x(t)),\Ps\left(\x(t),\sum_{i=1}^n\nabla f_i(\x(t))\right)\rangle].
    \end{align}

    To bound each term in \eqref{eqn:4inner_products}, \beh{using the inequality $(a+b)^2\leq 2(a^2+b^2)$}, we have
    \begin{align}
    \label{eqn:bound_of_nabla_f}
        \E[|\|\nabla f(\x(t))\|^2]=&\E[|\|\sum_{i=1}^n\nabla f_i(\x(t))\|^2]
        \leq 2\E[\|\sum_{i=1}^n(\nabla f_i(\x(t))-\pfi(t))\|^2]+2\E[\|\sum_{i=1}^n\pfi(t)\|^2]\cr
        \stackrel{\rm{(a)}}{\leq}& 2\E[(\sum_{i=1}^n\|\nabla f_i(\x(t))-\pfi(t)\|)^2]+2\E[\|\sum_{i=1}^n\grad_i(\x(t))\|^2]
        \stackrel{\rm{(b)}}{\leq} 2n^2(q^2(t)+G),
    \end{align}
    where (a) follows from \beh{triangle inequality for norms and $\|\E[\bv]\|^2\leq\E[\|\bv\|^2]$ for any random vector $\bv$.} In addition, note that \beh{$\pm \langle \bu,\bv\rangle\leq \frac{1}{2}(a \|\bu\|^2+\frac{1}{a}\|\bv\|^2)$ for any $a>0$ and for all vectors $\bu,\bv\in \R^d$. }
    
    Thus, by \eqref{eqn:bound_of_nabla_f}, \eqref{eqn:Ps_prop2}, and Lemma \ref{lemma:bound_of_g}, the first term in \eqref{eqn:4inner_products} can be bounded as
    \begin{align*}
        &\E[-\langle\nabla f(\x(t)),\Ps(\x(t),\grad(t))-\Ps\left(\x(t),\sum_{i=1}^n\pfi(\tauit)\right)\rangle]\\
        \leq&\frac{1}{2}\E[|\|\nabla f(\x(t))\|^2]+\frac{1}{2}\E[\|\Ps(\x(t),\grad(t))-\Ps\left(\x(t),\sum_{i=1}^n\pfi(\tauit)\right)\|^2]\\
        \leq&\frac{1}{2}\E[|\|\nabla f(\x(t))\|^2]+\frac{1}{2}\E[\|\grad(t)-\sum_{i=1}^n\pfi(\tauit)\|^2]\\
        \leq&n^2q^2(t)+n^2G+\frac{1}{2}n^2G
        =n^2q^2(t)+\frac{3}{2}n^2G.
    \end{align*}

    By \eqref{eqn:bound_of_nabla_f}, \eqref{eqn:Ps_prop2}, and Lemma \ref{lemma:effect_of_delay}, the second term in \eqref{eqn:4inner_products} can be bounded as
    \begin{align*}
        &\E[-\langle\nabla f(\x(t)),\Ps\left(\x(t),\sum_{i=1}^n\pfi(\tauit)\right)-\Ps\left(\x(t),\sum_{i=1}^n\pfi(t)\right)\rangle]\\
        \leq&\frac{p(t)}{2}\E[|\|\nabla f(\x(t))\|^2]+\frac{1}{2p(t)}\E[\|\Ps\left(\x(t),\sum_{i=1}^n\pfi(\tauit)\right)-\Ps\left(\x(t),\sum_{i=1}^n\pfi(t)\right)\|^2]\\
        \leq&\frac{p(t)}{2}\E[|\|\nabla f(\x(t))\|^2]+\frac{1}{2p(t)}\E[\|\sum_{i=1}^n\pfi(\tauit)-\sum_{i=1}^n\pfi(t)\|^2]\\
        \leq&n^2p(t)q^2(t)+n^2Gp(t)+\frac{1}{2}n^4GCL^2p(t).
    \end{align*}

    Using a similar argument as above, by \eqref{eqn:bound_of_nabla_f}, \eqref{eqn:Ps_prop2}, and Assumption \ref{asm:grad_est}, the third term in \eqref{eqn:4inner_products} can be bounded as
    \begin{align*}
        &\E[-\langle\nabla f(\x(t)),\Ps\left(\x(t),\sum_{i=1}^n\pfi(t)\right)-\Ps\left(\x(t),\sum_{i=1}^n\nabla f_i(\x(t))\right)\rangle]\\
        \leq&\frac{1}{2}\E[|\|\nabla f(\x(t))\|^2]+\frac{1}{2}\E[\|\Ps\left(\x(t),\sum_{i=1}^n\pfi(t)\right)-\Ps\left(\x(t),\sum_{i=1}^n\nabla f_i(\x(t))\right)\|^2]\\
        \leq&\frac{1}{2}\E[|\|\nabla f(\x(t))\|^2]+\frac{1}{2}\E[\|\sum_{i=1}^n\pfi(t)-\sum_{i=1}^n\nabla f_i(\x(t))\|^2]\\
        \leq&n^2q^2(t)+n^2G+\frac{1}{2}n^2q^2(t)
        =\frac{3}{2}n^2q^2(t)+n^2G.
    \end{align*}

    By \eqref{eqn:Ps_prop1}, the last term in \eqref{eqn:4inner_products} can be bounded as
    \begin{align*}
        \E[-\langle\nabla f(\x(t)),\Ps\left(\x(t),\sum_{i=1}^n\nabla f_i(\x(t))\right)\rangle]
        \leq-\E[\|\Ps(\x(t),\nabla f(\x(t)))\|^2]=-\E[\|\pg(t)\|^2].
    \end{align*}

    Then, \eqref{eqn:expected_iter} becomes
    \begin{align*}
        \E[&f(\x(t+1))] 
        \leq\E[f(\x(t))]-\step(t)\E[\|\pg(t)\|^2]\\
        &+\frac{5}{2}n^2G\step(t)+\frac{1}{2}n^2GL\step^2(t)+n^2G(1+\frac{1}{2}n^2CL^2)\step(t)p(t)+\frac{5}{2}n^2\step(t)q^2(t)+n^2\step(t)p(t)q^2(t).
    \end{align*}

    By taking the telescopic sum over $t=0,\ldots, T$ of the above inequality, we get
    \begin{align*}
        \sum_{t=0}^T \step(t)\E[\| \pg(t)\|^2]
        \leq \E[f(\x(0))]-f^*+\frac{5}{2}n^2G\sum_{t=0}^T\step(t)+\frac{1}{2}n^2GL\sum_{t=0}^T\step^2(t)\\
        +n^2G(1+\frac{1}{2}n^2CL^2)\sum_{t=0}^T\step(t)p(t)+\frac{5}{2}n^2\sum_{t=0}^T\step(t)q^2(t)+n^2\sum_{t=0}^T\step(t)p(t)q^2(t).
    \end{align*}

    Let $V(T):=\sum_{t=0}^T \step(t)\E[\| \pg(t)\|^2]$. Since $p(t)=\step(\lceil\kappa t\rceil)=\Oc\left(\step(t)\right)$, we have 
    \begin{align}
    \label{eqn:VT}
        V(T)=\Oc\left(\sum_{t=0}^T\left(\step(t)+\step^2(t)+\step(t)q^2(t)\right)\right).
    \end{align}

    Next, we establish an upper bound for $M(T):=\frac{1}{T+1}\sum_{t=0}^T \E[\| \pg(t)\|^2]$.
    For any $\theta \in (0,1)$,  define
    \begin{align*}
        M_\theta(T)=\left[\frac{1}{T+1}\sum_{t=0}^T \left(\E[\| \pg(t)\|^2]\right)^\theta\right]^\frac{1}{\theta}.
    \end{align*}
    Note that by H\"older's inequality (Theorem 6.2 in \cite{folland1999real}), for any $p,q>1$ with ${\frac{1}{p}+\frac{1}{q}=1}$, and non-negative sequences $\{a_t\}_{t=0}^T$ and $\{b_t\}_{t=0}^T$, we have
    \begin{align}
    \label{eqn:holder}
        \left(\sum_{t=0}^Ta_tb_t\right)^q \leq \left(\sum_{t=0}^T{a_t}^p\right)^{\frac{q}{p}}\sum_{t=0}^T{b_t}^q.
    \end{align}
    Let $a_t=\left( \frac{1}{\step(t)} \right)^\theta$, $b_t=\left(\step(t)\E[\| \pg (t)\|^2]\right)^\theta$, and $(p,q)=\left(\frac{1}{1-\theta},\frac{1}{\theta}\right)$. Noting  $\sum_{t=0}^Tb_t^{\frac{1}{\theta}}=V(T)$, and using~\eqref{eqn:holder}, we have
    \begin{align*}
        M_\theta(T)&=\left(\frac{1}{T+1}\sum_{t=0}^T{a_tb_t} \right)^\frac{1}{\theta}
        \leq (T+1)^{-\frac{1}{\theta}}\left(\sum_{t=0}^T a_t ^\frac{1}{1-\theta} \right)^\frac{1-\theta}{\theta} \sum_{t=0}^T b_t ^\frac{1}{\theta}\\
        &=(T+1)^{-\frac{1}{\theta}}\left(\sum_{t=0}^T \left(\frac{1}{\step(t)}\right) ^\frac{\theta}{1-\theta} \right)^\frac{1-\theta}{\theta}V(T).
    \end{align*}
    Using a similar argument as in the proof of Proposition 1 in \cite{reisizadeh2022dimix}, it can be shown that for any $\theta \in(0,1)$,
    \begin{align*}
        \sum_{t=0}^T \E[\| \pg(t)\|^2] \leq 
        \left[\sum_{t=0}^T \left(\E[\| \pg(t)\|^2]\right)^\theta\right]^\frac{1}{\theta}.
    \end{align*}
    Now we can bound $M(T)$ as
    \begin{align}
    \label{eqn:bound_for_MT}
        M(T) &\leq (T+1)^{\frac{1}{\theta}-1}M_\theta(T)
        \leq (T+1)^{-1}\left(\sum_{t=0}^T \left(\frac{1}{\step(t)}\right) ^\frac{\theta}{1-\theta} \right)^\frac{1-\theta}{\theta}V(T).
    \end{align}
    Noting that $\{\step(t)\}$ is a diminishing sequence, we have
    \begin{align*}
        \lim_{\theta\rightarrow1}\left(\sum_{t=0}^T \left(\frac{1}{\step(t)}\right) ^\frac{\theta}{1-\theta} \right)^\frac{1-\theta}{\theta}
        =\max_{0\leq t\leq T}\frac{1}{\step(t)}=\frac{1}{\step(T)}.
    \end{align*}
    Since (\ref{eqn:bound_for_MT}) holds for all $\theta\in(0,1)$,
        $M(T)=\Oc\left(T^{-1}\frac{1}{\step(T)}V(T)\right)$.
    Finally, recalling the bound of $V(T)$ in (\ref{eqn:VT}), we have
    \begin{align*}
        \frac{1}{T+1}\!\sum_{t=0}^{T}{\E[\|\pg(t)\|^2]}=\Oc\left(\frac{\sum_{t=0}^T\left(\step(t)+\step^2(t)+\step(t)q^2(t)\right)}{T\step(T)}\right).
    \end{align*}
\end{proof}

Our second result establishes convergence guarantees for strongly convex cost functions $f$.

\begin{theorem}\label{thm:stronglyconvex}
For $\mu$-strongly convex function $f$, i.e., $f(\y)\geq f(\x)+\langle\nabla f(\x), \y-\x\rangle+\frac{\mu}{2}\|\y-\x\|^2$ for all $\x,\y\in\Ss$, let $\x^*=\argmin_{\x\in\Ss}f(\x)$. Under Assumptions \ref{asm:func}-\ref{asm:delay}, if $\{\step(t)\}$ is a decreasing sequence s.t. $\lim_{t\rightarrow\infty}\frac{\step(\lceil\kappa t\rceil)}{\step(t)}=\bar{\eta}>0$ where $\bar{\eta}>0$ is a constant, $\sum_{t=0}^\infty \step(t)=\infty$, ${\sum_{t=0}^\infty\step^2(t)<\infty}$, and $\sum_{t=0}^\infty\step(t)q^2(t)<\infty$, then we have ${\lim_{t\rightarrow\infty}\E[\|\x(t)-\x^*\|^2]=0}$. 

Moreover, if $\step(t)=\frac{\step_0}{t+1}$ for some constant $\step_0>0$, and $q(t)=\frac{q_0}{(t+1)^\beta}$ for some constant $q_0>0$ and $\beta\geq\frac{1}{2}$, then $\E[\|\x(t)-\x^*\|^2]=\Oc(\frac{1}{t})$. 
\end{theorem}

\begin{proof}
(\textit{Proof sketch of Theorem \ref{thm:stronglyconvex}})

    We first establish a recursive inequality for $v(t):=\E[\|\x(t)-\x^*\|^2]$.

    From (\ref{eqn:maindyn}), using the nonexpansive property of projection (see Corollary 2.2.3 in~\cite{nesterov2018lectures}), we have
    \begin{align*}
        \|\x(t+1)-\x^*\|^2=&\|\proj[\x(t)-\step(t)\grad(t)]-\proj[\x^*]\|^2
        \leq \|\x(t)-\x^*-\step(t)\grad(t)\|^2\\
        =&\|\x(t)-\x^*\|^2+\step^2(t)\|\grad(t)\|^2-2\step(t)\langle\x(t)-\x^*,\grad(t)\rangle.
    \end{align*}
    The following arguments closely follow the proof of Theorem 1 in \cite{zheng2024general}, so we defer the full details to the appendix.
\end{proof}

\begin{corollary}
\label{cor:stronglyconvex_const_step}
For $\mu$-strongly convex function $f$, under Assumptions \ref{asm:func}-\ref{asm:delay}, if for all $t$, $\step(t)=\step$ for $\step\in(0,\frac{1}{\mu})$ and $q(t)=q>0$, then we have
\begin{align*}
    &\E[\|\x(t)-\x^*\|^2]\leq(1-\mu\step)^t \E[\|\x(0)-\x^*\|^2]
    +n^2G(1+2\sqrt{C})\frac{\step}{\mu}+2n^4CGL^2\frac{\step^2}{\mu^2}+2n^2\frac{q^2}{\mu^2}.
\end{align*}
If for all $t$, $\step(t)=\step$ for $\step\in(0,\frac{1}{\mu})$, and $q(t)=\frac{q_0}{(t+1)^\beta}$ for $\beta>0$, we have
\begin{align*}
    \E[\|\x(t)-\x^*\|^2]
    =\Oc(t^{-2\beta})+n^2G(1+2\sqrt{C})\frac{\step}{\mu}+2n^4CGL^2\frac{\step^2}{\mu^2}.
\end{align*}
\end{corollary}

\begin{remark}
    First, under the weaker delay assumption (Assumption~\ref{asm:delay}), Theorem \ref{thm:stronglyconvex} recovers the SGD-optimal rate $\Oc(\frac{1}{T})$ obtained in the bounded-delay result \cite{zheng2024general}. Second, Corollary~\ref{cor:stronglyconvex_const_step} reproduces the rate of Theorem~2 in \cite{zheng2024general} for constant step sizes and constant bias, and further shows that with constant step sizes and diminishing bias the mean squared error converges to a neighborhood whose radius is independent of the specific bias sequence $\{q(t)\}$; the convergence is sub-exponential yet super-polynomial.
\end{remark}

Our last result studies the main dynamics~\eqref{eqn:maindyn} for convex (but not necessarily strongly convex) functions $f$.
\begin{theorem}\label{thm:convex}
For a convex function $f$, in addition to Assumptions \ref{asm:func}-\ref{asm:delay}, assume that there exists $R\in(0,\infty)$ s.t. ${\E[\|\x(t)-\x^*\|^2]\leq R}$ for $\x^*\in \argmin_{\x\in\Ss}f(\x)$ and all $t$ (which is naturally present in all practical cases if $\Ss$ is bounded). Let $\{\step(t)\}$ be a decreasing sequence s.t. $\lim_{t\rightarrow\infty}\frac{\step(\lceil\kappa t\rceil)}{\step(t)}=\bar{\eta}>0$ where $\bar{\eta}>0$ is a constant and $\sum_{t=0}^{\infty}\step(t)=\infty$. Let $\Tilde{\x}(T)=\sum_{t=0}^T \frac{\step(t)}{\sum_{t=0}^T \step(t)}\x(t)$, then we have
\begin{align*}
    \E[f(\Tilde{\x}(T))]-f^*=\Oc\left(\frac{\sum_{t=0}^T \left(\step^2(t)+\step(t)q(t)\right)}{\sum_{t=0}^T \step(t)}\right).
\end{align*}
In particular, if $\step(t)=\frac{\step_0}{\sqrt{t+1}}$ for some constant $\step_0>0$, $q(t)=\frac{q_0}{(t+1)^\beta}$ for some constant $q_0>0$ and $\beta\geq\frac{1}{2}$, then $\E[f(\Tilde{\x}(T))]-f^*=\Oc\left(\frac{\log T}{\sqrt{T}}\right)=\Oc\left(\frac{1}{T^{\frac{1}{2}-\eps}}\right)$ for any $0<\eps<\frac{1}{2}$.
\end{theorem}

\begin{remark}
    Note that ${\E[\|\x(t)-\x^*\|^2]\leq R}$ is a mild assumption; in particular, Assumption 2.3 in the delay-adaptive work~\cite{sra2016adadelay} is strictly stronger. Theorem~\ref{thm:convex} establishes that, for convex objectives, our rate $\Oc\left(\frac{\log T}{\sqrt{T}}\right)$ matches the optimal rate for classical SGD~\cite{garrigos2023handbook} and differs from the delay-adaptive method of \cite{sra2016adadelay} by at most a logarithmic factor. 
\end{remark}

\section{Conclusion}

In conclusion, we develop a unified framework for distributed stochastic optimization under delayed and approximate (potentially biased) gradient models, showing that a simple pre-chosen diminishing step size is sufficient to attain near-optimal performance under a mild scaled-delay condition. Our results recover the classical SGD rates for nonconvex and strongly convex objectives and match delay-adaptive guarantees for convex objectives up to logarithmic factors, thereby clarifying that delay adaptivity is not necessary to achieve optimal convergence behavior in this setting. Beyond revisiting the role of step size selection, the analysis provides a clean characterization of how bias, stochasticity, constraints, and scaled delays interact, offering practical guidance for designing robust distributed learning systems and suggesting several directions for future work, such as tighter convex bounds without logarithmic terms and extensions to more general networked or fully decentralized architectures.

\acks{This research is supported by NSF grant CNS-2148313, AutoCOMBOT
MURI grant, AFOSR FA9550-23-1-0057 grant, and AFOSR FA9550-24-1-0129 grant.}

\bibliography{bib}

\appendix


\section{Proof of Lemma \ref{lemma:bound_of_g}}

\begin{proof}
 Using the expression \eqref{eqn:grad_t} for $\grad(t)$ and Assumption~\ref{asm:grad_est} we have
    \begin{align*}
        \E[\|\grad(t)\|^2]
        =\E\left[\left\|\sum_{i=1}^n\grad_i(\x(\tauit))\right\|^2\right]
        \leq n\sum_{i=1}^n\E\left[ \|\grad_i(\x(\tauit))\|^2\right]
        \leq n^2G.
    \end{align*}
\end{proof}

\section{Proof of Theorem \ref{thm:stronglyconvex}}

\begin{proof}
    We first establish a recursive inequality for $v(t):=\E[\|\x(t)-\x^*\|^2]$.

    From (\ref{eqn:maindyn}), using the nonexpansive property of projection (see Corollary 2.2.3 in~\cite{nesterov2018lectures}), we have
    \begin{align*}
        \|\x(t+1)-\x^*\|^2=&\|\proj[\x(t)-\step(t)\grad(t)]-\proj[\x^*]\|^2
        \leq \|\x(t)-\x^*-\step(t)\grad(t)\|^2\\
        =&\|\x(t)-\x^*\|^2+\step^2(t)\|\grad(t)\|^2-2\step(t)\langle\x(t)-\x^*,\grad(t)\rangle.
    \end{align*}
    Taking the expected value of both sides, and using Lemma \ref{lemma:bound_of_g}, we have
    \begin{align*}
        v(t+1)\leq v(t)+n^2G\step^2(t)-2\step(t)\E[\langle\x(t)-\x^*,\grad(t)\rangle].
    \end{align*}
    Let us rewrite the last term on the right hand side of the above inequality as
    \begin{align}
    \label{eqn:expected_innerprod}
        &-\E[\langle\x(t)-\x^*,\grad(t)\rangle]
        =-\E\left[\langle\x(t)-\x^*,\sum_{i=1}^n\grad_i(\x(\tauit))\rangle\right]\cr
        =&-\sum_{i=1}^n\E\left[\langle\x(t)-\x(\tauit),\grad_i(\x(\tauit))-\pfi(\tauit)\rangle\right]
        -\sum_{i=1}^n\E\left[\langle\x(\tauit)-\x^*,\grad_i(\x(\tauit))-\pfi(\tauit)\rangle\right]\cr
        &-\sum_{i=1}^n\E\left[\langle\x(t)-\x^*,\pfi(\tauit)-\pfi(t)\rangle\right]
        -\sum_{i=1}^n\E\left[\langle\x(t)-\x^*,\pfi(t)-\nabla f_i(\x(t))\rangle\right]\cr
        &-\sum_{i=1}^n\E\left[\langle\x(t)-\x^*,\nabla f_i(\x(t))\rangle\right].
    \end{align}


    To bound the first term in (\ref{eqn:expected_innerprod}), using Cauchy-Schwarz inequality and the first inequality in  Lemma~\ref{lemma:effect_of_delay}, we have
    \begin{align*}
        &-\sum_{i=1}^n\E\left[\langle\x(t)-\x(\tauit),\grad_i(\x(\tauit))-\pfi(\tauit)\rangle\right]
        \leq \sum_{i=1}^n\E\left[\|\x(t)-\x(\tauit)\|\|\grad_i(\x(\tauit))-\pfi(\tauit)\|\right]\\
        &\leq \sum_{i=1}^n\sqrt{\E\left[\|\x(t)-\x(\tauit)\|^2\right]}
        \cdot \sqrt{\E\left[\|\grad_i(\x(\tauit))-\pfi(\tauit)\|^2\right]}\\
        &\leq \sum_{i=1}^n\sqrt{n^2GCp^2(t)\cdot G}
        =n^2\sqrt{C}Gp(t).
    \end{align*}
    For the second term in (\ref{eqn:expected_innerprod}), using the law of total expectation, we have
    \begin{align*}
        &-\sum_{i=1}^n\E\left[\langle\x(\tauit)-\x^*,\grad_i(\x(\tauit))-\pfi(\tauit)\rangle\right]\\
        =&-\sum_{i=1}^n\E\left[\E\left[\langle\x(\tauit)-\x^*, \grad_i(\x(\tauit))-\pfi(\tauit)\rangle | \sigma(\x(\tauit))\right]\right]\\
        =&-\sum_{i=1}^n\E\left[\langle\x(\tauit)-\x^*, \E\left[\grad_i(\x(\tauit))-\pfi(\tauit) | \sigma(\x(\tauit))\right]\rangle\right]=0.
    \end{align*}
    To bound the third term in (\ref{eqn:expected_innerprod}), using $ab\leq\frac{1}{2}(a^2+b^2)$ and Lemma \ref{lemma:effect_of_delay}, we have
    \begin{align*}
        &-\sum_{i=1}^n\E\left[\langle\x(t)-\x^*,\pfi(\tauit)-\pfi(t)\rangle\right]\\
        \leq&\frac{1}{2}\sum_{i=1}^n\bigg(\frac{\mu}{2n}\E[\|\x(t)-\x^*\|^2]
        +\frac{2n}{\mu}\E[\|\pfi(\tauit)-\pfi(t)\|^2]\bigg)\leq\frac{1}{2}\left(\frac{\mu}{2}v(t)+2n^4CGL^2\frac{p^2(t)}{\mu}\right).
    \end{align*}
    To bound the fourth term in (\ref{eqn:expected_innerprod}), using $ab\leq\frac{1}{2}(a^2+b^2)$ and Assumption \ref{asm:grad_est}, we have
    \begin{align*}
        &-\sum_{i=1}^n\E\left[\langle\x(t)-\x^*,\pfi(t)-\nabla f_i(\x(t))\rangle\right]\\
        \leq&\frac{1}{2}\sum_{i=1}^n\left(\frac{\mu}{2n}\E[\|\x(t)-\x^*\|^2]
        +\frac{2n}{\mu}\E[\|\pfi(t)-\nabla f_i(\x(t))\|^2]\right)
        \leq \frac{1}{2}\left(\frac{\mu}{2}v(t)+2n^2\frac{q^2(t)}{\mu}\right).
    \end{align*}
    To bound the last term in (\ref{eqn:expected_innerprod}), using the strong convexity of $f$, we have
    \begin{align*}
        &-\sum_{i=1}^n\E\left[\langle\x(t)-\x^*,\nabla f_i(\x(t))\rangle\right]
        =-\E[\langle\x(t)-\x^*,\nabla f(\x(t))\rangle]
        \leq-\mu\E[\|\x(t)-\x^*\|^2]
        =-\mu v(t).
    \end{align*}
    
   Using the above inequalities to bound the terms in~\eqref{eqn:expected_innerprod}, we have the following recursive inequality for $v(t)$ 
    \begin{align}
    \label{eqn:vt_recursion}
        v(t+1)\leq&(1-\mu\step(t))v(t)+n^2G\step^2(t)+2n^2\sqrt{C}G\step(t)p(t)
        +2n^4CGL^2\frac{\step(t)p^2(t)}{\mu}+2n^2\frac{\step(t)q^2(t)}{\mu}.
    \end{align}
    Since $p(t)=\Oc(\step(t))$, if $\sum_{t=0}^\infty\step(t)=\infty$, ${\sum_{t=0}^\infty\step^2(t)<\infty}$, and ${\sum_{t=0}^\infty\step(t)q^2(t)<\infty}$, using the deterministic version of Robbins-Siegmund Theorem \cite{robbins1971convergence}, we conclude that 
    \[{\lim_{t\rightarrow\infty}v(t)=\lim_{t\rightarrow\infty}\E[\|\x(t)-\x^*\|^2]=0}.\]

    If we choose $\step(t)=\frac{\step_0}{t+1}$ for some $\step_0>0$, and $q(t)=\frac{q_0}{(t+1)^\beta}$ for some constant $q_0>0$ and $\beta\geq\frac{1}{2}$, it follows by Lemma 5 in \cite{reisizadeh2022distributed} or a similar argument in the derivations in Section 2.1 in \cite{nemirovski2009robust} that ${v(t)=\E[\|\x(t)-\x^*\|^2]=\Oc(\frac{1}{t})}$.
\end{proof}

\section{Proof of Corollary \ref{cor:stronglyconvex_const_step}}

\begin{proof}
    If we let $\step(t)=\step$ and $q(t)=q$ for all $t$, \eqref{eqn:vt_recursion} becomes
    \begin{align*}
        v(t+1)\leq& (1-\mu\step)v(t)+n^2G(1+2\sqrt{C})\step^2
        +2n^4CGL^2\frac{\step^3}{\mu} +2n^2\frac{\step q^2}{\mu}.
    \end{align*}
    Let $r(\step,q):=n^2G(1+2\sqrt{C})\step^2+2n^4CGL^2\frac{\step^3}{\mu}+2n^2\frac{\step q^2}{\mu}$. Then, when $\mu\step<1$, we have
    \begin{align*}
        v(t)&\leq (1-\mu\step)^t v(0)+r(\step,q)\sum_{s=0}^{t-1}(1-\mu\step)^s
        \leq (1-\mu\step)^t v(0)+r(\step,q)\frac{1}{\mu\step}.
    \end{align*}
    As a result, when $0<\step<\frac{1}{\mu}$, $\x(t)$ converges to the $r$-neighborhood of $\x^*$ exponentially fast where ${r=r(\step,q)\frac{1}{\mu\step}=n^2G(1+2\sqrt{C})\frac{\step}{\mu}+2n^4CGL^2\frac{\step^2}{\mu^2}+2n^2\frac{q^2}{\mu^2}}$.

    If we let $\step(t)=\step$ for all $t$ and $\{q(t)\}$ be a diminishing sequence in the form of $q(t)=\frac{q_0}{(t+1)^\beta}$, \eqref{eqn:vt_recursion} becomes
    \begin{align*}
        v(t+1)\leq& (1-\mu\step)v(t)+n^2G(1+2\sqrt{C})\step^2
        +2n^4CGL^2\frac{\step^3}{\mu} +2n^2\frac{\step}{\mu}q^2(t).
    \end{align*}
    Let $r'(\step):=n^2G(1+2\sqrt{C})\step^2+2n^4CGL^2\frac{\step^3}{\mu}$. Then, when $\mu\step<1$, we have
    \begin{align*}
        v(t)\leq& (1-\mu\step)^t v(0)+r'(\step)\sum_{s=0}^{t-1}(1-\mu\step)^s
        +2n^2\frac{\step}{\mu}\sum_{s=0}^{t-1} (1-\mu\step)^{t-1-s}q^2(s)\\
        \leq& (1-\mu\step)^t v(0)+r'(\step)\frac{1}{\mu\step}+2n^2\frac{\step}{\mu}\sum_{s=0}^{t-1} (1-\mu\step)^{t-1-s}q^2(s).
    \end{align*}
    By Lemma 5 in \cite{reisizadeh2022distributed}, we have $\sum_{s=0}^{t-1} (1-\mu\step)^{t-1-s}q^2(s)=\Oc(t^{-2\beta})$. Therefore, when $0<\step<\frac{1}{\mu}$, $\x(t)$ converges to the $r'$-neighborhood of $\x^*$ where $r'=r'(\step)\frac{1}{\mu\step}=n^2G(1+2\sqrt{C})\frac{\step}{\mu}+2n^4CGL^2\frac{\step^2}{\mu^2}$.

\end{proof}

\section{Proof of Theorem \ref{thm:convex}}

\begin{proof}
    Choose an $x^*\in\argmin_{\x\in\Ss}f(\x)$. We still employ the dynamic for $v(t)=\E[\|\x(t)-\x^*\|^2]$ and bound each term in \eqref{eqn:expected_innerprod} respectively, but treat them in a slightly different manner.
    
    We bound the first two terms in \eqref{eqn:expected_innerprod} in the same way of proving Theorem \ref{thm:stronglyconvex}.

    Using Cauchy-Schwarz inequality, Lemma \ref{lemma:effect_of_delay}, and the assumption of $\E[\|\x(t)-\x^*\|^2]\leq R$ for all $t$, we can bound the third term in \eqref{eqn:expected_innerprod} as
    \begin{align*}
        &-\sum_{i=1}^n\E\left[\langle\x(t)-\x^*,\pfi(\tauit)-\pfi(t)\rangle\right]\\
        \leq&\sum_{i=1}^n\E[\|\x(t)-\x^*\|\cdot\|\pfi(\tauit)-\pfi(t)\|]\\
        \leq&\sum_{i=1}^n\sqrt{\E[\|\x(t)-\x^*\|^2]}\cdot\sqrt{\E[\|\pfi(\tauit)-\pfi(t)\|^2]}
        \leq n^2\sqrt{RGC}Lp(t).
    \end{align*}

    Using Cauchy-Schwarz inequality, Assumption \ref{asm:grad_est}, and the assumption of $\E[\|\x(t)-\x^*\|^2]\leq R$ for all $t$, we can bound the fourth term in \eqref{eqn:expected_innerprod} as
    \begin{align*}
        &-\sum_{i=1}^n\E\left[\langle\x(t)-\x^*,\pfi(t)-\nabla f_i(\x(t))\rangle\right]\\
        \leq&\sum_{i=1}^n\E[\|\x(t)-\x^*\|\cdot\|\pfi(t)-\nabla f_i(\x(t))\|]\\
        \leq&\sum_{i=1}^n\sqrt{\E[\|\x(t)-\x^*\|^2]}\cdot\sqrt{\E[\|\pfi(t)-\nabla f_i(\x(t))\|^2]}
        \leq n\sqrt{R}q(t).
    \end{align*}

    Using the convexity of $f$, we can bound the last term in \eqref{eqn:expected_innerprod} as 
    \begin{align*}
        &-\sum_{i=1}^n\E\left[\langle\x(t)-\x^*,\nabla f_i(\x(t))\rangle\right]
        =-\E[\langle\x(t)-\x^*,\nabla f(\x(t))\rangle]
        \leq f(\x^*)-\E[f(\x(t))].
    \end{align*}

    Using the above inequalities to bound the terms in~\eqref{eqn:expected_innerprod}, we have the following bound for $\step(t)\left(\E[f(\x(t))]-f(\x^*)\right)$:
    \begin{align*}
        2\step(t)\left(\E[f(\x(t))]-f(\x^*)\right)
        \leq &v(t)-v(t+1)+n^2G\step^2(t)+2n^2\sqrt{CG}(\sqrt{G}+L\sqrt{R})\step(t)p(t)\\
        &+2n\sqrt{R}\step(t)q(t).
    \end{align*}
    By taking the telescoping sum, we get
    \begin{align*}
        2\sum_{t=0}^T\step(t)\left(\E[f(\x(t))]-f(\x^*)\right)
        \leq &v(0)+n^2G\sum_{t=0}^T\step^2(t)+2n^2\sqrt{CG}(\sqrt{G}+L\sqrt{R})\sum_{t=0}^T\step(t)p(t)\\
        &+2n\sqrt{R}\sum_{t=0}^T\step(t)q(t).
    \end{align*}
    Dividing both sides by $2\sum_{t=0}^T \step(t)$, we have
    \begin{align*}
        \E\left[\frac{\sum_{t=0}^T\step(t)f(\x(t))}{\sum_{t=0}^T \step(t)}\right]-f(\x^*)
        \leq &\frac{1}{2\sum_{t=0}^T \step(t)}\left(v(0)+n^2G\sum_{t=0}^T\step^2(t)\right.\\
        &\left.+2n^2\sqrt{CG}(\sqrt{G}+L\sqrt{R})\sum_{t=0}^T\step(t)p(t)+2n\sqrt{R}\sum_{t=0}^T\step(t)q(t)\right).
    \end{align*}
    Let $\Tilde{\x}(T)=\sum_{t=0}^T \frac{\step(t)}{\sum_{t=0}^T \step(t)}\x(t)$. Since $f$ is convex, we have
    \begin{align*}
        \E[f(\Tilde{\x}(T))]-f^*\leq \E\left[\frac{\sum_{t=0}^T\step(t)f(\x(t))}{\sum_{t=0}^T \step(t)}\right]-f^*=\Oc\left(\frac{\sum_{t=0}^T \left(\step^2(t)+\step(t)q(t)\right)}{\sum_{t=0}^T \step(t)}\right).
    \end{align*}

\end{proof}

\end{document}